\documentclass[11pt]{amsart}

\usepackage{amsmath}
\usepackage{amssymb}
\usepackage{amsthm}
\usepackage{bm}
\usepackage{geometry}
\usepackage{color}
\usepackage{hyperref}
\usepackage{enumerate}
\hypersetup{colorlinks=true}

\newtheorem{theorem}{Theorem}[section]

\newtheorem{lemma}[theorem]{Lemma}
\theoremstyle{remark}

\title[Zeckendorf representation of $(a^{-1} \bmod F_n)$]
{Zeckendorf representation of multiplicative inverses modulo a Fibonacci number}

\author[G.~Alecci]{Gessica Alecci}
\address{\parbox{\linewidth}{
Department of Mathematical Sciences, Politecnico di Torino\\
Corso Duca degli Abruzzi 24, 10129 Torino, Italy\\[-8pt]}}
\email{gessica.alecci@polito.it}

\author[N.~Murru]{Nadir Murru}
\address{\parbox{\linewidth}{
Department of Mathematics, Universit\`a degli Studi di Trento\\
Via Sommarive 14, I-38123 Povo (Trento), Italy\\[-8pt]}}
\email{nadir.murru@unitn.it}

\author[C.~Sanna]{Carlo Sanna}
\address{\parbox{\linewidth}{
Department of Mathematical Sciences, Politecnico di Torino\\
Corso Duca degli Abruzzi 24, 10129 Torino, Italy\\[-8pt]}}
\email{carlo.sanna.dev@gmail.com}

\begin{document}

\begin{abstract}
Prempreesuk, Noppakaew, and Pongsriiam determined the Zeckendorf representation of the multiplicative inverse of $2$ modulo $F_n$, for every positive integer $n$ not divisible by~$3$, where $F_n$ denotes the $n$th Fibonacci number. 
We determine the Zeckendorf representation of the multiplicative inverse of $a$ modulo $F_n$, for every fixed integer $a \geq 3$ and for all positive integers $n$ with $\gcd(a, F_n) = 1$.
Our proof makes use of the so-called base-$\varphi$ expansion of real numbers.
\end{abstract}

\maketitle

\section{Introduction}

Let $(F_n)_{n \geq 1}$ be the sequence of Fibonacci numbers, which is defined by the initial conditions $F_1=F_2=1$ and by the linear recurrence $F_n=F_{n-1} + F_{n-2}$ for $n \geq 3$.
It is well known~\cite{Z72} that every positive integer $n$ can be written as a sum of distinct non-consecutive Fibonacci numbers, that is, $n = \sum_{i=1}^m d_i F_i$, where $m \in \mathbb{N}$, $d_i \in \{0, 1\}$, and $d_i d_{i+1} = 0$ for all $i \in \{1, \ldots, m-1\}$. 
This is called the \emph{Zeckendorf representation} of $n$ and, apart from the equivalent use of $F_1$ instead of $F_2$ or vice versa, is unique.

The Zeckendorf representation of integer sequences has been studied in several works. 
For~instance, Filipponi and Freitag~\cite{FF89, FF93} studied the Zeckendorf representation of numbers of the form $F_{kn}/F_n$, $F_n^2/d$ and $L_n^2/d$, where $L_n$ are the Lucas numbers and $d$ is a Lucas or Fibonacci number.
Filipponi, Hart, and Sanchis~\cite{Filipponi, Hart, HS99} analyzed the Zeckendorf representation of numbers of the form $mF_n$.
Filipponi~\cite{Filipponi} determined the Zeckendorf representation of $m F_n F_{n + k}$ and $mL_n L_{n + k}$ for $m \in \{1,2,3,4\}$. 
Bugeaud~\cite{B21} studied the Zeckendorf representation of smooth numbers.
The study of Zeckendorf representations has been also approached from a combinatorial point of view~\cite{AR09, Gerdemann, GR11, W07}. 
Moreover, generalizations of the Zeckendorf representation to linear recurrences other than the sequence of Fibonacci numbers has been considered~\cite{D60, DD14, GT90, GT94, PT89}. 

For all integers $a$ and $m \geq 1$ with $\gcd(a, m) = 1$, let $(a^{-1} \bmod m)$ denote the least positive multiplicative inverse of $a$ modulo $m$, that is, the unique $b \in \{1, \dots, m\}$ such that ${ab \equiv 1} \pmod m$.
Prempreesuk, Noppakaew, and Pongsriiam~\cite{MR4057296} determined the Zeckendorf representation of $(2^{-1} \bmod F_n)$, for every positive integer $n$ that is not divisible by~$3$.
(The condition $3 \nmid n$ is necessary and sufficient to have $\gcd(2, F_n) = 1$.)
In~particular, they showed~\cite[Theorem~3.2]{MR4057296} that
\begin{equation*}
(2^{-1} \bmod F_n) = 
\begin{cases} 
\sum_{k \,=\, 0}^{(n - 7) / 2} F_{n - 3k - 2} + F_3 & \text{ if } n \equiv 1 \pmod 3 ; \\[5pt]
\sum_{k \,=\, 0}^{(n - 8) / 2} F_{n - 3k - 2} + F_4 & \text{ if } n \equiv 2 \pmod 3 ;
\end{cases}
\end{equation*}
for every integer $n \geq 8$.
We extend their result by determining the Zeckendorf representation of the multiplicative inverse of $a$ modulo $F_n$, for every fixed integer $a \geq 3$ and every positive integer $n$ with $\gcd(a, F_n) = 1$.
Precisely, we prove the following result.

\begin{theorem}\label{thm:main}
Let $a \geq 3$ be an integer.
Then there exist integers $M, n_0, i_0 \geq 1$ and periodic sequences $\bm{z}^{(0)}, \dots, \bm{z}^{(M - 1)}$ and $\bm{w}^{(1)}, \dots, \bm{w}^{(i_0)}$ with values in $\{0, 1\}$ such that, for all integers $n \geq n_0$ with $\gcd(a, F_n) = 1$, the Zeckendorf representation of $(a^{-1} \bmod F_n)$ is given by
\begin{equation*}
(a^{-1} \bmod F_n) = \sum_{i \,=\, i_0}^{n - 1} z_{n - i}^{(n \bmod M)} F_i + \sum_{i \,=\, 1}^{i_0 - 1} w_n^{(i)} F_i .
\end{equation*}
\end{theorem}
From the proof of Theorem~\ref{thm:main} it follows that $M, n_0, i_0$, $\bm{z}^{(0)}, \dots, \bm{z}^{(M - 1)}$, and $\bm{w}^{(1)}, \dots, \bm{w}^{(i_0)}$ are effectively computable in terms of $a$.

\subsection*{Acknowledgments}

The authors are members of GNSAGA of INdAM and of CrypTO, the group of Cryptography and Number Theory of Politecnico di Torino.

\section{Preliminaries on Fibonacci numbers}

Let us recall that for every integer $n \geq 1$ it holds the \emph{Binet formula}
\begin{equation*}\label{equ:Binet}
F_n = \frac{\varphi^n - \overline{\varphi}^n}{\sqrt{5}} ,
\end{equation*}
where $\varphi := (1 + \sqrt{5}) / 2$ is the Golden ratio and $\overline{\varphi} := (1 - \sqrt{5}) / 2$ is its algebraic conjugate.
Furthermore, it is well known that for every integer $m \geq 1$ the Fibonacci sequence $(F_n)_{n \geq 1}$ is (purely) periodic modulo $m$.
Let $\pi(m)$ denote its period length, or the so-called \emph{Pisano period}.

The next lemma gives a formula for the inverse of $a$ modulo $F_n$.

\begin{lemma}\label{lem:inverse}
For all integers $a \geq 1$ and $n \geq 3$ with $\gcd(a, F_n) = 1$, we have that
\begin{equation*}
(a^{-1} \bmod F_n) = \frac{b F_n + 1}{a} ,
\end{equation*}
where $b := (-F_r^{-1} \bmod a)$ and $r := (n \bmod \pi(a))$.
\end{lemma}
\begin{proof}
Since $r \equiv n \pmod {\pi(a)}$, we have that $F_r \equiv F_n \pmod a$.
In particular, it follows that $\gcd(a, F_r) = \gcd(a, F_n) = 1$.
Hence, $F_r$ is invertible modulo $a$, and consequently $b$ is well defined.
Moreover, we have that
\begin{equation*}
bF_n + 1 \equiv -F_r^{-1} F_r + 1 \equiv 0 \pmod a ,
\end{equation*}
and thus $c := (bF_n + 1) / a$ is an integer.
On the one hand, we have that
\begin{equation*}
ac \equiv b F_n + 1 \equiv 1 \pmod {F_n} .
\end{equation*}
On the other hand, since $b \leq a - 1$ and $n \geq 3$, we have that
\begin{equation*}
0 \leq c \leq \frac{(a - 1) F_n + 1}{a} = F_n - \frac{F_n - 1}{a} < F_n .
\end{equation*}
Therefore, we get that $c = (a^{-1} \bmod F_n)$, as desired.
\end{proof}

\section{Preliminaries on base-$\varphi$ expansion}
We need some basic results regarding the so-called \emph{base-$\varphi$ expansion} of real numbers, which was introduced by Bergman~\cite{MR131384} in 1957 (see also~\cite{MR1573104}), and which is a particular case of non-integer base expansion (see, e.g.,~\cite{MR142719}).
Let $\mathfrak{D}$ be the set of sequences in $\{0, 1\}$ that have no two consecutive terms equal to $1$, and that are not ultimately equal to the periodic sequence $0, 1, 0, 1, \dots$.
Then for every $x \in {[0, 1)}$ there exists a unique sequence $\bm{\delta}(x) = (\delta_i(x))_{i \in \mathbb{N}}$ in $\mathfrak{D}$ such that $x = \sum_{i = 1}^\infty \delta_i(x) \varphi^{-i}$.
Precisely, $\delta_i(x) = \lfloor T^{(i)}(x) \rfloor$ for every $i \in \mathbb{N}$, where $T^{(i)}$ denotes the $i$th iterate of the map $T : {[0, 1)} \to {[0, 1)}$ defined by $T(\hat{x}) := (\varphi \hat{x} \bmod 1)$ for every $\hat{x} \in {[0, 1)}$.
Furthermore, letting $\mathcal{F} := \mathbb{Q}(\varphi) \cap {[0,1)}$, if $x \in \mathcal{F}$ then $\bm{\delta}(x)$ is ultimately periodic.
In~particular, if $x \in \mathcal{F}$ is given as $x = x_1 + x_2 \varphi$, where $x_1, x_2 \in \mathbb{Q}$, then the preperiod and the period of $\bm{\delta}(x)$ can be effectively computed by finding the smallest $i \in \mathbb{N}$ such that $T^{(i)}(x) = T^{(j)}(x)$ for some $j \in \mathbb{N}$ with $j < i$.
Conversely, for every ultimately periodic sequence $\bm{d} = (d_i)_{i \in \mathbb{N}}$ in $\mathfrak{D}$ we have that the number $x = \sum_{i=1}^\infty d_i \varphi^{-i}$ belongs to $\mathcal{F}$, and $x_1, x_2 \in \mathbb{Q}$ such that $x = x_1 + x_2 \varphi$ can be effectively computed in terms of the preperiod and period of $\bm{d}$ by using the formula for the sum of the geometric series. Moreover, in the case that $x$ is a rational number in $[0, 1)$ then $\bm{\delta}(x)$ is purely periodic \cite{S80}. 

The next lemma collects two easy inequalities for sums involving sequences in $\mathfrak{D}$.

\begin{lemma}\label{lem:easy-sums}
For every sequence $(d_i)_{i \in \mathbb{N}}$ in $\mathfrak{D}$ and for every $m \in \mathbb{N} \cup \{\infty\}$, we have:
\begin{enumerate}
\item\label{lem:easy-sums:1} $\sum_{i = 1}^m d_i \varphi^{-i} \in {[0, 1)}$ and
\item\label{lem:easy-sums:2} $\sum_{i = 1}^m d_i (-\varphi)^{-i} \in {(-1, \varphi^{-1})}$.
\end{enumerate}
\end{lemma}
\begin{proof}
Since $(d_i)_{i \in \mathbb{N}}$ belongs to $\mathfrak{D}$, there exists $k \in \mathbb{N}$ such that $d_k = d_{k + 1} = 0$.
Let $k$ be the minimum integer with such property.
Then
\begin{align*}
\sum_{i = 1}^\infty d_i \varphi^{-i} &= \sum_{i = 1}^{k - 1} d_i \varphi^{-i} + \sum_{i = k + 2}^{\infty} d_i \varphi^{-i} < \sum_{j = 1}^{\lfloor k / 2\rfloor} \varphi^{-(2j - 1)} + \sum_{i = k + 2}^{\infty} \varphi^{-i} \\
    &= \left(1 - \varphi^{-2 \lfloor k / 2 \rfloor}\right) + \varphi^{-k} \leq 1 ,
\end{align*}
and~(\ref{lem:easy-sums:1}) is proved.
Let us prove~(\ref{lem:easy-sums:2}).
On the one hand, we have
\begin{equation*}
\sum_{i = 1}^m d_i (-\varphi)^{-i} \leq \sum_{j = 1}^m d_{2j} \varphi^{-2j} < \sum_{j = 1}^\infty \varphi^{-2j} = \varphi^{-1} ,
\end{equation*}
where the second inequality is strict because $\mathfrak{D}$ does not contain sequences that are ultimately equal to $(0, 1, 0, 1, \dots)$.
On the other hand, similarly, we have
\begin{equation*}
\sum_{i = 1}^m d_i (-\varphi)^{-i} \geq -\sum_{j = 1}^m d_{2j - 1} \varphi^{-(2j - 1)} > -\sum_{j = 1}^\infty \varphi^{-(2j - 1)} = -1 .
\end{equation*}
Thus~(\ref{lem:easy-sums:2}) is proved.
\end{proof}

The following lemma relates base-$\varphi$ expansion and Zeckendorf representation.

\begin{lemma}\label{lem:phi-to-zeck}
Let $N$ be a positive integer and write $N = x \varphi^m / \sqrt{5}$ for some $x \in \mathcal{F}$ and some integer $m \geq 2$.
Then the Zeckendorf representation of $N$ is given by
\begin{equation*}
N = \sum_{i = 1}^{m - 1} \delta_{m - i}(x) F_i .
\end{equation*}
Moreover, we have $\delta_m(x) = 0$.
\end{lemma}
\begin{proof}
Let $R := N - \sum_{i = 1}^{m - 1} \delta_{m - i} (x) F_i$.
We have to prove that $R = 0$.
Since $R$ is an integer, it suffices to show that $|R| < 1$.
We have
\begin{align*}
\sqrt{5} N &= x \varphi^m = \sum_{i = 1}^\infty \delta_i(x) \varphi^{m - i} = \sum_{i = 1}^m \delta_i(x) \varphi^{m - i} + \sum_{i = m + 1}^\infty \delta_i(x) \varphi^{m - i} \\
 &= \sum_{i = 0}^{m - 1} \delta_{m - i}(x) \varphi^i + \sum_{i = 1}^\infty \delta_{i + m}(x) \varphi^{-i} \\
 &= \sum_{i = 0}^{m - 1} \delta_{m - i}(x) (\varphi^i - \overline{\varphi}^i) + \sum_{i = 0}^{m - 1} \delta_{m - i}(x) \overline{\varphi}^i + \sum_{i = 1}^\infty \delta_{i + m}(x) \varphi^{-i} \\
 &= \sqrt{5} \sum_{i = 1}^{m - 1} \delta_{m - i}(x) F_i + \sum_{i = 0}^{m - 1} \delta_{m - i}(x) (-\varphi)^{-i} + \sum_{i = 1}^\infty \delta_{i + m}(x) \varphi^{-i} .
\end{align*}
Hence, we get that
\begin{equation*}
\sqrt{5} R = \sum_{i = 0}^{m - 1} \delta_{m - i}(x) (-\varphi)^{-i} + \sum_{i = 1}^\infty \delta_{i + m}(x) \varphi^{-i} .
\end{equation*}

For the sake of contradiction, suppose that $\delta_m(x) = 1$. Then $\delta_{m+1}(x) = 0$ and, by Lemma~\ref{lem:easy-sums}, it follows that
\begin{equation*}
\sqrt{5} R = 1 + \sum_{i = 1}^{m - 1} \delta_{m - i}(x) (-\varphi)^{-i} + \sum_{i = 2}^\infty \delta_{i + m}(x) \varphi^{-i} \in (1 - 1 + 0, 1 + \varphi^{-1} + \varphi^{-1}) = (0, \sqrt{5}) ,
\end{equation*}
which is a contradiction, since $R$ is an integer.

Therefore, $\delta_m(x) = 0$ and, again by Lemma~\ref{lem:easy-sums}, we have
\begin{equation*}
\sqrt{5} R = \sum_{i = 1}^{m - 1} \delta_{m - i}(x) (-\varphi)^{-i} + \sum_{i = 1}^\infty \delta_{i + m}(x) \varphi^{-i} \in (-1 + 0, \varphi^{-1} + 1) \subseteq (-\sqrt{5}, \sqrt{5}) ,
\end{equation*}
so that $|R| < 1$, as desired.
\end{proof}

The next two lemmas regards the base-$\varphi$ expansions of the sum of two numbers.

\begin{lemma}\label{lem:base-phi-sum}
Let $x, y \in {[0, 1)}$, $m \in \mathbb{N}$, and put $v := x + y \varphi^{-m}$.
Suppose that there exists $\lambda \in \mathbb{N}$ such that $\lambda + 2 \leq m$ and $\delta_\lambda(x) = \delta_{\lambda + 1}(x) = 0$.
Then, putting
\begin{equation*}
w := \sum_{i = \lambda + 2}^\infty \delta_i(x) \varphi^{-i} + \sum_{i = m + 1}^\infty \delta_{i - m}(y) \varphi^{-i} ,
\end{equation*}
we have that $v, w \in {[0, 1)}$ and
\begin{equation}\label{equ:digits-sum}
\delta_i(v) = \begin{cases}
    \delta_i(x) &\text{ if } i \leq \lambda , \\
    \delta_i(w) &\text{ if } i > \lambda ,
    \end{cases}
\end{equation}
for every $i \in \mathbb{N}$.
\end{lemma}
\begin{proof}
From Lemma~\ref{lem:easy-sums}(\ref{lem:easy-sums:1}), we have that
\begin{equation*}
0 \leq w < \varphi^{-(\lambda + 1)} + \varphi^{-m} < \varphi^{-(\lambda + 1)} + \varphi^{-(\lambda + 2)} = \varphi^{-\lambda} .
\end{equation*}
Hence, $w \in {[0,\varphi^{-\lambda})} \subseteq {[0, 1)}$ and so $w = \sum_{i = \lambda + 1}^\infty \delta_i(w) \varphi^{-i}$.
Therefore, recalling that $\delta_{\lambda + 1}(x) = 0$, we get that
\begin{align*}
v &= x + y\varphi^{-m} = \sum_{i = 1}^\infty \delta_i(x) \varphi^{-i} + \sum_{i = 1}^\infty \delta_i(y) \varphi^{-i-m} = \sum_{i = 1}^\infty \delta_i(x) \varphi^{-i} + \sum_{i = m + 1}^\infty \delta_{i-m}(y) \varphi^{-i} \\
    &= \sum_{i = 1}^{\lambda} \delta_i(x) \varphi^{-i} + w = \sum_{i = 1}^{\lambda} \delta_i(x)\varphi^{-i} + \sum_{i = \lambda + 1}^{\infty} \delta_i(w)\varphi^{-i} ,
\end{align*}
which is the base-$\varphi$ expansion of $v$. 
(Note that $\delta_\lambda(x) = 0$.)
In~particular, by Lemma~\ref{lem:easy-sums}(\ref{lem:easy-sums:1}), we have that $v \in {[0, 1)}$.
Thus~\eqref{equ:digits-sum} follows.
\end{proof}

\section{Proof of Theorem~\ref{thm:main}}

Fix an integer $a \geq 3$.
Let us begin by defining $M, n_0, i_0$, and $\bm{z}^{(0)}, \dots, \bm{z}^{(M - 1)}$.
Put $M := \pi(a)$.
For each $r \in \{0, \dots, M - 1\}$ with $\gcd(a, F_r) = 1$, let $b_r := (-F_r^{-1} \bmod a)$, $x_r := b_r / a$, and $\bm{z}^{(r)} := \bm{\delta}(x_r)$.
Note that $x_r \in (0, 1)$.
Since $x_r$ is a positive rational number, we have that $\bm{z}^{(r)}$ is a (purely) periodic sequence belonging to $\mathfrak{D}$.
Let $\ell$ be the least common multiple of the period lengths of $\bm{z}^{(0)}, \dots, \bm{z}^{(M - 1)}$, and put $i_0 := \ell + 3$.
Finally, let $n_0 := \max\{i_0 + 1, \lceil \log(2a) / \!\log\varphi\rceil\}$.

Pick an integer $n \geq n_0$ with $\gcd(a, F_n) = 1$ and, for the sake of brevity, put $r := (n \bmod M)$.
From Lemma~\ref{lem:inverse} and Binet's formula~\eqref{equ:Binet}, we get that
\begin{equation}\label{equ:proof1}
(a^{-1} \bmod F_n) = \frac{b_r F_n + 1}{a} = \frac{b_r(\varphi^n - \overline{\varphi}^n)}{\sqrt{5} a} + \frac1{a} = (x_r + y_n \varphi^{-n}) \frac{\varphi^n}{\sqrt{5}} ,
\end{equation}
where 
\begin{equation*}
y_n := \frac{\sqrt{5}}{a} - x_r (-\varphi)^{-n} .
\end{equation*}
Since $n \geq n_0$, it follows that $y_n \in (0, 1)$ and $x_r + y_n \varphi^{-n} \in (0, 1)$ .
Therefore, from~\eqref{equ:proof1} and Lemma~\ref{lem:phi-to-zeck}, we get that
\begin{equation*}
(a^{-1} \bmod F_n) = \sum_{i \,=\, 1}^{n - 1} \delta_{n - i}(x_r + y_n \varphi^{-n}) F_i .
\end{equation*}
Since $\bm{\delta}(x_r)$ is (purely) periodic and belongs to $\mathfrak{D}$, we have that $\bm{\delta}(x_r)$ contains infinitely many pairs of consecutive zeros.
Furthermore, since the period length of $\bm{\delta}(x_r)$ is at most $\ell$, we have that among every $\ell+1$ consecutive terms of $\bm{\delta}(x_r)$ there are two consecutive zero.
In~particular, there exists $\lambda = \lambda(r)$ such that $n - \ell - 3 \leq \lambda \leq n - 2$ and $\delta_{\lambda}(x_r) = \delta_{\lambda + 1}(x_r) = 0$.
Consequently, by Lemma~\ref{lem:base-phi-sum}, we get that $\delta_i(x_r + y_n \varphi^{-n}) = \delta_i(x_r)$ for each positive integer $i \leq \lambda$ and, a fortiori, for each positive integer $i \leq n - i_0$.
Therefore, we have that
\begin{align}\label{equ:proof2}
(a^{-1} \bmod F_n) &= \sum_{i \,=\, i_0}^{n - 1} \delta_{n - i}(x_r) F_i + \sum_{i \,=\, 1}^{i_0 - 1} \delta_{n - i}(x_r + y_n \varphi^{-n}) F_i \\
    &= \sum_{i \,=\, i_0}^{n - 1} z_{n - i}^{(r)} F_i + \sum_{i \,=\, 1}^{i_0 - 1} w_n^{(i)} F_i , \nonumber 
\end{align}
where $\bm{w}^{(1)}, \cdots, \bm{w}^{(i_0)}$ are the sequences defined by $w_n^{(i)} := \delta_{n - i}(x_r + y_n \varphi^{-n})$.
Note that, by construction,
\begin{equation*}
z_1^{(r)}, z_2^{(r)}, \dots, z_{n - i_0}^{(r)}, w_n^{(i_0-1)}, w_n^{(i_0-2)}, \dots, w_n^{(1)}
\end{equation*}
is a string in $\{0,1\}$ with no consecutive zeros.
Hence, \eqref{equ:proof2} is the Zeckendorf representation of $(a^{-1} \bmod F_n)$.

It remains only to prove that $\bm{w}^{(1)}, \cdots, \bm{w}^{(i_0)}$ are periodic.
By~\eqref{equ:proof2} and the uniqueness of the Zeckendorf representation, it suffices to prove that
\begin{equation}\label{equ:proof3}
R(n) := (a^{-1} \bmod F_n) - \sum_{i \,=\, i_0}^{n - 1} z_{n - i}^{(r)} F_i = \sum_{i \,=\, 1}^{i_0 - 1} w_n^{(i)} F_i .
\end{equation}
is a periodic function of $n$.
From the last equality in~\eqref{equ:proof3}, we have that $0 \leq R(n) < \sum_{i \,=\, 1}^{i_0 - 1} F_i$.
(Actually, one can prove that $0 \leq R(n) < F_{i_0}$, but this is not necessary for our proof.)
Fix a prime number $p > \max\{a, \sum_{i \,=\, 1}^{i_0 - 1} F_i\}$.
It suffices to prove that $R(n)$ is periodic modulo $p$.
Recalling that $(a^{-1} \bmod F_n) = (b_r F_n + 1) / a$ and that the sequence of Fibonacci numbers is periodic modulo $p$, it follows that $(a^{-1} \bmod F_n)$ is periodic modulo $p$.
Hence, it suffices to prove that $R^\prime(n) := \sum_{i = i_0}^{n - 1} z_{n - i}^{(r)} F_i$ is periodic modulo $p$.
Using that $\bm{z}^{(r)}$ has period length dividing $\ell$, we get that
\begin{align*}
R^\prime(n + \ell M) - R^\prime(n) &= \sum_{i \,=\, i_0}^{n + \ell M - 1} z_{n + \ell M - i}^{((n + \ell M) \bmod M)} F_i - \sum_{i = i_0}^{n - 1} z_{n - i}^{(r)} F_i \\
    &= \sum_{i \,=\, i_0}^{n + \ell M - 1} z_{n + \ell M - i}^{(r)} F_i - \sum_{i \,=\, i_0}^{n - 1} z_{n - i}^{(r)} F_i \\
    &= \sum_{i \,=\, n}^{n + \ell M - 1} z_{n + \ell M - i}^{(r)} F_i + \sum_{i \,=\, i_0}^{n - 1} (z_{n + \ell M - i}^{(r)} - z_{n - i}^{(r)}) F_i \\
    &= \sum_{j \,=\, 1}^{\ell M} z_{j}^{(r)} F_{n + \ell M - j} ,
\end{align*}
which is a linear combination of sequences that are periodic modulo $p$.
Hence $R^\prime(n)$ is periodic modulo $p$.
The proof is complete.

\end{document}